\documentclass{amsart}
\usepackage{amscd, amsmath, amsthm, amssymb, mathdots}
\usepackage{enumerate, comment}
\usepackage{tikz}
\usepackage{pgfplots}
\usepackage[utf8]{inputenc}
\usepgfplotslibrary{ternary}
\usepackage{ifthen} 
\usepackage[bookmarksnumbered,colorlinks, plainpages]{hyperref}
\usepackage[capitalize,noabbrev]{cleveref} 
\usepackage{stackrel}
\hypersetup{
pdftitle={Research Paper},
pdfauthor={Reza Abdolmaleki},
pdfsubject={Minimal generators of powers of monomial ideals},
pdfcreator={Reza Abdolmaleki},
colorlinks,
linkcolor=blue,
citecolor=magenta,
anchorcolor=red,
bookmarksopen,
urlcolor=black,
filecolor=red,
}
\theoremstyle{plain}
\newtheorem{Theorem}{Theorem}[section]

\newtheorem{Corollary}[Theorem]{Corollary}
\newtheorem{Proposition}[Theorem]{Proposition}

\theoremstyle{definition}

\newtheorem{Example}[Theorem]{Example}

\newtheorem{Question}[Theorem]{Question}
\newtheorem{Fact}[Theorem]{Fact}

\theoremstyle{remark}
\newtheorem*{Acknowledgments}{Acknowledgments}

\newcommand{\ZZ}{\mathbb{Z}}

\def\cocoa{{\hbox{\rm C\kern-.13em o\kern-.07em C\kern-.13em o\kern-.15em A}}}

\newcommand{\fkm}{\mathfrak{m}}

\DeclareMathOperator{\lcm}{lcm}
\DeclareMathOperator{\sign}{sign}

\begin{document}

\title{Prescribed Initial Behavior of $\mu(I^k)$}

\author{Reza Abdolmaleki}
\address{Department of Mathematics, Lahore University of Management Sciences, DHA, Lahore Cantt. 54792, Lahore, Pakistan}
\email{reza.abdolmaleki@lums.edu.pk}
\email{reza.abd110@gmail.com}

\author{Shinya Kumashiro}
\address{Department of Mathematics, Osaka Institute of Technology, 5-16-1 Omiya, asahi-ku, Osaka, 535-8585, Japan}
\email{shinya.kumashiro@oit.ac.jp}
\email{shinyakumashiro@gmail.com}

\begin{abstract}
It is well known that for  every graded ideal $I$, the numbers of minimal generators $\mu(I^k)$ of its powers of $I$ are eventually increasing. However, its initial behavior can be surprisingly flexible. We prove that any prescribed finite pattern of increases, decreases, and equalities can occur among the first differences $\mu(I^{k+1})-\mu(I^k)$ of a suitable monomial ideal $I$ in $K[x,y]$. This provides a broad positive answer to a previously posed sign-realization problem and unifies several known constructions exhibiting unusual behavior of $\mu(I^k)$. Moreover, as a consequence, analogous sign-realization results are obtained for the index of reducibility of powers of $\mathfrak m$-primary ideals in two-dimensional regular local rings.
\end{abstract}


\thanks{}

\maketitle

\section{Introduction}

The behavior of algebraic invariants under taking powers of ideals is a
central theme in commutative algebra.  Many such invariants are eventually
governed by polynomial or linear functions, and this asymptotic regularity is
often well understood.  However, the initial behavior before the asymptotic
range begins can be much more subtle.  It may reflect phenomena which are
invisible from the eventual polynomial behavior alone.

One basic invariant of this kind is the number $\mu(I^k)$  of minimal generators of powers
of a graded ideal $I$.  If \(I\) is a graded ideal, then the function
\[
 k\longmapsto \mu(I^k)
\]
is the Hilbert function of the fiber cone of \(I\).  Hence it agrees with the
Hilbert polynomial of the fiber cone for all sufficiently large \(k\) (\cite[Theorem 4.1.3]{BH}).  In
particular, the sequence \(\{\mu(I^k)\}_{k\ge 1}\) is eventually increasing.
Thus the natural question is not whether monotonicity eventually occurs, but
rather how complicated the sequence can be before it enters its eventual
asymptotic regime.

The first examples showing that the initial behavior can be far from
monotone were constructed in several forms. It was shown in
\cite{EHS} that there exist monomial ideals with tiny squares, that is,
ideals satisfying $ \mu(I^2)<\mu(I)$. 
Later, examples with tiny higher powers were constructed in \cite{G}, where
$ \mu(I^k)<\mu(I)$ for suitable values of \(k\).
 
Further examples of non-monotone behavior and a systematic study of the initial behavior of the sequence $\{\mu(I^k)\}_{k\ge 1}$ were given in \cite{AHZ}, where it was shown that in two variables the sequence $\mu(I^k)$ can attain a local maximum or a local minimum at an arbitrary power $k$. Moreover, for any given number $q$, there exists an ideal in $2\ell$ variables, where $\ell = (q+1)^2$, having at least $q$ local maxima.

The next construction was given in \cite{AK}. It was shown
that, for every positive integer \(n\), there exists a monomial ideal
\(I\) such that
\[
\mu(I^{n})<\mu(I^{n-1})<\cdots <\mu(I).
\]

 These works show that the first few values of \(\mu(I^k)\) are not constrained by the eventual monotonicity in any simple way.

In \cite[Question~2.8(a)]{AK}, the authors asked whether one can
prescribe arbitrarily the signs of the initial differences
\(\mu(I^{k+1})-\mu(I^k)\) for a monomial ideal \(I\subset K[x,y]\),
allowing only increases and decreases.
In this paper, we consider the following more general problem, where
zero differences are also allowed.

\begin{Question}(cf. \cite[Question~2.8(a)]{AK}) \label{q11}
Can one prescribe arbitrarily the signs of the initial differences
\(\mu(I^{k+1})-\mu(I^k)\), allowing increases, decreases, and equalities,
before \(\mu(I^k)\) becomes linear? In other words, given an integer
\(m\ge 3\) and signs
\[
 \varepsilon_1,\dots,\varepsilon_{m-2}\in\{+,-,0\},
\]
does there exist a monomial ideal \(I\subset K[x,y]\) such that
\[
 \sign(\mu(I^{k+1})-\mu(I^k))=\varepsilon_k
 \qquad
 (1\le k\le m-2)
\]
and
\[
 \mu(I^k)=(m+2)k+1
 \qquad
 (k\ge m)?
\]
Here, the sign function is defined by
\[
\sign(t)=
\begin{cases}
+ & \text{if } t>0,\\
- & \text{if } t<0,\\
0 & \text{if } t=0.
\end{cases}
\]
\end{Question}

\vspace{1em}

We prove that this generalized sign-realization problem has a positive
answer in the polynomial ring \(K[x,y]\). In fact, for every
\(m\ge 3\) and every choice of signs
\[
 \varepsilon_1,\dots,\varepsilon_{m-2}\in\{+,-,0\},
\]
we construct a monomial ideal \(I\subset K[x,y]\) satisfying the two
conditions in Question~\ref{q11} (Theorem~\ref{thm:prescribed-signs}). 
Moreover, by choosing the parameters in the construction sufficiently large,
we also obtain the full sign-realization statement in the case where the last
prescribed sign is negative (Corollary~\ref{cor:terminal-negative}).

Thus our result gives a unified way to produce monomial ideals whose
sequences of minimal numbers of generators exhibit prescribed initial behavior. Especially, it covers all cases provided in \cite{AHZ, AK, EHS, G}. 
Moreover, allowing zero signs produces arbitrarily long initial plateaus.  In particular, for
every positive integer \(t\), one obtains monomial ideals for which
\[
 \mu(I)=\mu(I^2)=\cdots=\mu(I^t)
\]
while the sequence eventually becomes increasing.  These examples show that
the initial behavior of \(\{\mu(I^k)\}_{k\ge 1}\) is highly flexible. 

\vspace{1em}

Our results also have consequences for the index of reducibility of powers
of ideals in Noetherian local rings. Recall that, if \(J\) is an
\(\mathfrak m\)-primary ideal in a Noetherian local ring
\((R,\mathfrak m)\), then
\[
r(R/J)=\ell_R\bigl((J:_R\mathfrak m)/J\bigr),
\]
where \(\ell_R(-)\) denotes length, is called the {\it index of reducibility}
of \(J\), or equivalently the Cohen--Macaulay type of \(R/J\).
It is known that the function $n\mapsto r(R/J^n)$ is eventually polynomial of degree $\dim R-1$ (\cite[Theorem~4.1]{CQT}). Moreover, in a two-dimensional regular local ring, $r(R/J^n)=\mu(J^n)-1$, for all $n\ge 1$ \cite[Proposition~2.9]{AK}. Consequently, our sign-realization results for $\mu(I^n)$ yield analogous realization results for the index of reducibility. In particular, arbitrary prescribed patterns of increases, decreases, and equalities can also occur among the consecutive differences $r(R/J^{n+1})-r(R/J^n)$ before the sequence reaches its eventual polynomial behavior (\cref{cor:Index-red}).


\section{Main theorem}

In what follows, let $S=K[x, y]$ be the polynomial ring over a field $K$ and $I$ a monomial ideal of $S$. 
The main goal of this article is to prove the following.

\begin{Theorem}\label{thm:prescribed-signs}
Let \(m\ge 3\), and let $  \varepsilon_1,\dots,\varepsilon_{m-2}\in\{+,-,0\}$. 
Then there exists a monomial ideal \(I \subset S\) such that
\[
 \sign(\mu(I^{k+1})-\mu(I^k))=\varepsilon_k
 \qquad \text{ for all $1\le k\le m-2$}
\]
and
\[
 \mu(I^k)=(m+2)k+1
 \qquad \text{ for all $k\ge m$}.
\]
\end{Theorem}

To construct certain monomial ideals satisfying Theorem \ref{thm:prescribed-signs}, we shall use the following result from \cite{AK}.

\begin{Fact}\label{thm:ABK-construction}{\rm (\cite[Theorem 2.1]{AK})}
Let $m, p_1, \dots, p_m, a_2, \dots, a_m\ge 2$ be positive integers such that 
\begin{center}
$p_1=(a_i+1)p_i$ for $i\in \{2,\dots, m\}$ \quad and \quad $p_2+\cdots +p_{m-1} < p_1$. 
\end{center}
Let 
\begin{align*}
I_1=&(x^{p_1}, y^{p_1})(x^{(m+1)p_1}, y^{(m+1)p_1}) \text{ and}\\
I_i=&x^{ip_1+p_i}{\cdot}y^{(m+2-i)p_1+p_2+\cdots+p_i}{\cdot}(x^{p_i}, y^{p_i})^{a_i-1}
\end{align*}
be ideals of $S$ for $i\in \{2, \dots, m\}$. Set $I=I_1 + \dots+I_m$. 
Then 
\[
I^k=\begin{cases}
I_1^{k-1}(I_1+\cdots+I_{m+1-k}) & \text{ if $1\le k \le m-1$}\\
I_1^k & \text{ if $m\le k$}
\end{cases}
\]
and  
\begin{equation}
\label{number}
\mu(I^k)=\begin{cases}
(k+1)^2+k(a_2+\cdots+a_{m+1-k}) & \text{ if $1\le k \le m-1$}\\
(m+2)k+1 & \text{ if $m\le k$}.
\end{cases}
\end{equation}
\end{Fact}

\vspace{1em}


By Fact~\ref{thm:ABK-construction}, we have 
\begin{equation}\label{eq:difference-formula}
\begin{aligned}
\mu(I^{k+1})-\mu(I^k)
&=(k+2)^2-(k+1)^2 +(k+1)(a_2+\cdots+a_{m-k})
-k(a_2+\cdots+a_{m+1-k})\\
&=2k+3+a_2+\cdots+a_{m-k}-k a_{m+1-k}
\end{aligned}
\end{equation}
for all $k\in \{1, \dots, m-2\}$. 
Set
\begin{equation}\label{eq:vk-definition}
 v_k=\mu(I^{k+1})-\mu(I^k)-2k-3
 \qquad
 (1\le k\le m-2).
\end{equation}
Therefore, 
\begin{equation}\label{eq:vk-linear-expression}
 v_k=a_2+\cdots+a_{m-k}-k a_{m+1-k}
 \qquad
 (1\le k\le m-2).
\end{equation}

We now rewrite \eqref{eq:vk-linear-expression} as a matrix equation.
Let
\[
 \mathbf a=
 \begin{pmatrix}
 a_2\\
 a_3\\
 \vdots\\
 a_m
 \end{pmatrix}
\quad
\text{and}
\quad 
 \mathbf v=
 \begin{pmatrix}
 v_1\\
 v_2\\
 \vdots\\
 v_{m-2}
 \end{pmatrix}.
\]
Define the \((m-2)\times(m-1)\) matrix \(A\) as
\[
A=
\begin{pmatrix}
1&1&\cdots&1&1&-1\\
1&1&\cdots&1&-2&0\\
\vdots&\vdots&\iddots&-3&0&0\\
1&1&\iddots&\iddots&\vdots&\vdots\\
1&-(m-2)&0&\cdots&0&0
\end{pmatrix}.
\]
With this notation, \eqref{eq:vk-linear-expression} is exactly $A\mathbf a=\mathbf v$. 
Therefore, in order to prescribe the signs of $\mu(I^{k+1})-~\mu(I^k)$ for $k\in \{1, \dots, m-2\}$, 
we need to find integer solutions \(\mathbf a\) for $A\mathbf a=\mathbf v$.

\begin{Proposition}\label{prop:linear-solution}
Let \(m\ge 3\), and put $L=\lcm(1,2,\dots,m-1)$.
Let $v_1,\dots,v_{m-2}$ be integers satisfying
\[
 v_r+2r+3\in L\ZZ
 \qquad
 (1\le r\le m-2).
\]
Define
\begin{align*}
\lambda_2=&
\left\lceil
\frac{2m(m-1)}{L}
\right\rceil \quad \text{and}\\
\lambda_j=&
\left\lceil
\frac{
(m-j+2)(v_{m-j+1}+m^2+(1-j)m+3)
+
\sum_{r=m-j+2}^{m-2}(v_r+2r+3)
}
{L}
\right\rceil
\end{align*}
for \(3\le j\le m\), where $\lceil * \rceil$ denotes the smallest integer greater than or equal to the number $*$.

Choose an integer \(\lambda\) such that
\[
 \lambda\ge \max_{2\le j\le m}\lambda_j.
\]
Define
\begin{align*}
 a_2=&
 \frac{\lambda L}{m-1}-m-1\quad \text{and}\\
 a_j=&
 1+
 \frac{
 \lambda L
 -(m-j+2)(v_{m-j+1}+2m-2j+5)
 -\sum_{r=m-j+2}^{m-2}(v_r+2r+3)
 }
 {(m-j+1)(m-j+2)}
\end{align*}
for \(3\le j\le m\). 
Then the following hold: 
\begin{enumerate}[\rm(a)] 
\item $a_2,a_3,\dots,a_m\in\ZZ$.
\item $a_i\ge m-1$ for all $2\le i\le m$.
\item $A\mathbf a=\mathbf v$. 
\end{enumerate}

\end{Proposition}

\begin{proof}
(a): Since $m-1$ divides $L$,
we have $a_2\in\ZZ$. 
Let \(3\le j\le m\).  Since the two integers $m-j+1$ and $m-j+2$ are relatively prime and both divide \(L\),
$(m-j+1)(m-j+2)$ divides $L$.
On the other hand, by the definition of $v_j$, 
\begin{align*}
v_r+2r+3  \quad (m-j+2 \le r \le m-2)
\end{align*}
are divisible by $(m-j+1)(m-j+2)$. Hence, noting that 
\[
v_{m-j+1}+2m-2j+5 = v_{m-j+1}+2(m-j+1)+3, 
\]
the numerator of $a_j-1$ 
is divisible by $(m-j+1)(m-j+2)$.
Thus $a_j\in\ZZ$ for all $3\le j\le m$.

(b): Since $\lambda\ge \lambda_2$,
we have $\lambda L\ge 2m(m-1)$.
Therefore
\[
 a_2
 =
 \frac{\lambda L}{m-1}-m-1
 \ge
 2m-m-1
 =
 m-1.
\]

Now let \(3\le j\le m\).  We first note that
\[
\begin{aligned}
m^2+(1-j)m+3
&=(m-2)(m-j+1)+2m-2j+5.
\end{aligned}
\]
Therefore, 
\[
\begin{aligned}
&(m-j+2)(v_{m-j+1}+m^2+(1-j)m+3)\\
=&
(m-j+2)(v_{m-j+1}+2m-2j+5)
+
(m-2)(m-j+1)(m-j+2).
\end{aligned}
\]
Since $ \lambda\ge \lambda_j$, the definition of \(\lambda_j\) gives
\[
\begin{aligned}
\lambda L
&\ge
(m-j+2)(v_{m-j+1}+m^2+(1-j)m+3)
+
\sum_{r=m-j+2}^{m-2}(v_r+2r+3)\\
&=
(m-j+2)(v_{m-j+1}+2m-2j+5)
+
(m-2)(m-j+1)(m-j+2) +\sum_{r=m-j+2}^{m-2}(v_r+2r+3).
\end{aligned}
\]
It follows that 
\begin{align*}
 a_j=&1+
 \frac{
 \lambda L
 -(m-j+2)(v_{m-j+1}+2m-2j+5)
 -\sum_{r=m-j+2}^{m-2}(v_r+2r+3)}{(m-j+1)(m-j+2)}\\
 \ge&
1+ \dfrac{(m-2)(m-j+1)(m-j+2)}{(m-j+1)(m-j+2)}\\
 =&m-1
\end{align*}
Thus, $ a_i\ge m-1$ for $i \in \{2, \ldots ,  m\}$.

(c): For \(3\le j\le m\), the definition of \(a_j\) gives
\begin{equation}\label{eq:linear-solution-aj}
\begin{aligned}
&(m-j+1)(m-j+2)(a_j-1)\\
=&
 \lambda L
 -(m-j+2)(v_{m-j+1}+2m-2j+5)
 -\sum_{r=m-j+2}^{m-2}(v_r+2r+3).
\end{aligned}
\end{equation}

Taking \(j=3\) in \eqref{eq:linear-solution-aj}, we obtain
\[
 (m-2)(m-1)(a_3-1)
 =
 \lambda L
 -(m-1)(v_{m-2}+2m-1).
\]
Since $ \lambda L=(m-1)(a_2+m+1)$, we get
\begin{align*}
 (m-2)(a_3-1)
 =&
 a_2+m+1-v_{m-2}-2m+1\\
 =&a_2-m+2-v_{m-2}.
\end{align*}
It follows that $a_2-(m-2)a_3=v_{m-2}$.
This is Equation \eqref{eq:vk-linear-expression} for \(k=m-2\).

Next let \(3\le j\le m-1\).  Applying \eqref{eq:linear-solution-aj} to
\(j+1\), we get
\[
\begin{aligned}
&(m-j)(m-j+1)(a_{j+1}-1)\\
=&
 \lambda L
 -(m-j+1)(v_{m-j}+2m-2j+3)
 -\sum_{r=m-j+1}^{m-2}(v_r+2r+3).
\end{aligned}
\]
Subtracting \eqref{eq:linear-solution-aj} from this equality gives
\[
\begin{aligned}
&(m-j)(m-j+1)(a_{j+1}-1)
-
(m-j+1)(m-j+2)(a_j-1)\\
=&
-(m-j+1)(v_{m-j}+2m-2j+3)
-\sum_{r=m-j+1}^{m-2}(v_r+2r+3)\\
&
+(m-j+2)(v_{m-j+1}+2m-2j+5)
+\sum_{r=m-j+2}^{m-2}(v_r+2r+3)\\
=&
-(m-j+1)(v_{m-j}+2m-2j+3)
+(m-j+2)(v_{m-j+1}+2m-2j+5)\\
&  -(v_{m-j+1}+2m-2j+5)\\
=&
-(m-j+1)(v_{m-j}+2m-2j+3)
+(m-j+1)(v_{m-j+1}+2m-2j+5)\\
=&
(m-j+1)
\bigl(
v_{m-j+1}-v_{m-j}+2
\bigr).
\end{aligned}
\]
Dividing by \(m-j+1\), we obtain
\[
 (m-j)(a_{j+1}-1)-(m-j+2)(a_j-1)
 =
 v_{m-j+1}-v_{m-j}+2.
\]
Equivalently,
\begin{align}\label{eqsugoi}
v_{m-j}=v_{m-j+1}+ (m-j+2)a_j-(m-j)a_{j+1}.
\end{align}

Now, we prove that Equation \eqref{eq:vk-linear-expression} holds for all \(1\le k\le m-2\). We have already proved this equation for \(k=m-2\). 
Assume that $k<m-2$ and Equation \eqref{eq:vk-linear-expression} holds for $k+1$. Applying \eqref{eqsugoi} to $j=m-k$, we obtain that
\begin{align*}
v_{k}=&v_{k+1} + (k+2)a_{m-k}-ka_{m-k+1}\\
=& \big( a_2+\cdots+a_{m-k-1}-(k+1) a_{m-k}\big)+ (k+2)a_{m-k}-ka_{m-k+1}\\
=&a_2+\cdots+a_{m-k}-k a_{m+1-k}.
\end{align*}
Hence, Equation \eqref{eq:vk-linear-expression} holds for $k$. By the descending induction, Equation \eqref{eq:vk-linear-expression} holds for all \(1\le k\le m-2\).
Thus $ A\mathbf a=\mathbf v$. This completes the proof.
\end{proof}

\begin{proof}[Proof of Theorem~\ref{thm:prescribed-signs}]
Put $ L=\lcm(1,2,\dots,m-1)$. For \(1\le r\le m-2\), define
\[
 v_r=
 \begin{cases}
 L-(2r+3),&\text{if }\varepsilon_r=+,\\
 -L-(2r+3),&\text{if }\varepsilon_r=-,\\
 -(2r+3),&\text{if }\varepsilon_r=0.
 \end{cases}
\]
Then
\[
 v_r+2r+3\in L\ZZ
 \qquad
 (1\le r\le m-2).
\]
By Proposition~\ref{prop:linear-solution}, there exist integers $ a_2,a_3,\dots,a_m$ such that $ a_i\ge m-1$ for $i \in \{2, \ldots ,  m\}$, and
$  A\mathbf a=\mathbf v$.

Now set $ p_1=\prod_{i=2}^{m}(a_i+1)$ and
\[
 p_i=\frac{p_1}{a_i+1}
 \qquad
 (2\le i\le m).
\]
Then
\[
 p_1=(a_i+1)p_i
 \qquad
 (2\le i\le m).
\]
Since $ a_i\ge m-1$, we have $a_i+1\ge m$ for $i \in \{2, \ldots ,  m\}$. Thus,
\[
 \frac{p_2+\cdots+p_{m-1}}{p_1}
 =
 \frac{1}{a_2+1}+\cdots+\frac{1}{a_{m-1}+1}
 \le
 \frac{m-2}{m}
 <1,
\]
and hence, $ p_2+\cdots+p_{m-1}<p_1$. Therefore, the hypotheses of Fact~\ref{thm:ABK-construction} are satisfied.

Let \(I\) be the monomial ideal obtained from Fact~\ref{thm:ABK-construction}.
By \eqref{eq:difference-formula}, for \(1\le r\le m-2\), we have
\[
 \mu(I^{r+1})-\mu(I^r)
 =
 2r+3+a_2+\cdots+a_{m-r}-r a_{m-r+1}.
\]
Since $ a_2+\cdots+a_{m-r}-r a_{m-r+1}=v_r$, we obtain $ \mu(I^{r+1})-\mu(I^r)=2r+3+v_r$.
By the definition of \(v_r\),
\[
 2r+3+v_r=
 \begin{cases}
 L,&\text{if }\varepsilon_r=+,\\
 -L,&\text{if }\varepsilon_r=-,\\
 0,&\text{if }\varepsilon_r=0.
 \end{cases}
\]
Therefore
\[
 \sign(\mu(I^{r+1})-\mu(I^r))=\varepsilon_r
 \qquad
 (1\le r\le m-2).
\]
Finally, Fact~\ref{thm:ABK-construction} gives
\[
 \mu(I^k)=(m+2)k+1
 \qquad
 (k\ge m).
\]
This completes the proof.
\end{proof}

\begin{Corollary}\label{cor:terminal-negative}
Let \(m\ge 3\), and let $\varepsilon_1,\dots,\varepsilon_{m-1}\in\{+,-,0\}$ with $\varepsilon_{m-1}=-$. Then there exists a monomial ideal \(I\subset K[x,y]\) such that
\[
\sign(\mu(I^{r+1})-\mu(I^r))=\varepsilon_r
\qquad
(1\le r\le m-1)
\]
and
\[
\mu(I^k)=(m+2)k+1
\qquad
(k\ge m).
\]
\end{Corollary}

\begin{proof}
Put $L=\lcm(1,2,\dots,m-1)$. Apply the construction in the proof of Theorem~\ref{thm:prescribed-signs} to $\varepsilon_1,\dots,\varepsilon_{m-2}$. 
In the choice of \(\lambda\) in Proposition~\ref{prop:linear-solution}, we
may impose the additional condition
\[
\lambda L>m(m+2),
\]
because increasing \(\lambda\) preserves all inequalities in
Proposition~\ref{prop:linear-solution}.

The signs of $\mu(I^{r+1})-\mu(I^r)$ for $r \in \{1, \ldots , m-2 \}$ are already prescribed by Theorem~\ref{thm:prescribed-signs}.  It remains
to compute $\mu(I^{m}) - \mu(I^{m-1})$.  By Fact~\ref{thm:ABK-construction},
\[
\mu(I^{m-1})=m^2+(m-1)a_2
\quad \text{and} \quad
\mu(I^m)=(m+2)m+1.
\]
Therefore, $ \mu(I^m)-\mu(I^{m-1})=(m+2)m+1-m^2-(m-1)a_2=2m+1-(m-1)a_2$.

Since $a_2=\frac{\lambda L}{m-1}-m-1$, we get
\[
\begin{aligned}
\mu(I^m)-\mu(I^{m-1})
=2m+1-
\left(\lambda L-(m-1)(m+1)\right)=m(m+2)-\lambda L.
\end{aligned}
\]
By the additional choice $\lambda L>m(m+2)$, this number is negative.  Hence
\[
\sign(\mu(I^m)-\mu(I^{m-1}))=
\varepsilon_{m-1}.
\]
This completes the proof.
\end{proof}


Using the construction developed above, we can produce examples exhibiting any prescribed initial behavior of the number of minimal generators of powers of monomial ideals. In particular, one can obtain an arbitrary number of consecutive constant values, an arbitrary number of decreases, and arbitrary numbers of local minima and local maxima. The following example presents three such behaviors.

\begin{Example}\label{examp}
\begin{enumerate}[{\rm (a)}]

\item We provide an ideal such that $\mu(I^k)$ is constant for $k=1,\ldots,4$ and increasing for $k\ge 5$. With the notation of \cref{prop:linear-solution,thm:ABK-construction}, let $m=4$ and $\lambda=4$. Then
\[
a_2=3,\qquad a_3=5,\qquad a_4=13.
\]
Hence,
\[
p_1=336,\qquad p_2=84,\qquad p_3=56,\qquad p_4=24.
\]
This yields
\[
\begin{aligned}
I_1 &= (x^{336},y^{336})(x^{1680},y^{1680}),\\
I_2 &= x^{756}y^{1428}(x^{84},y^{84})^2,\\
I_3 &= x^{1064}y^{1148}(x^{56},y^{56})^4,\\
I_4 &= x^{1368}y^{836}(x^{24},y^{24})^{12}.
\end{aligned}
\]
Let $I = I_1 + I_2 + I_3 + I_4$. Then, by \eqref{number}, we obtain
\[
\mu(I)=\mu(I^2)=\mu(I^3)=\mu(I^4)=25,
\]
and
\[
\mu(I^5)=31.
\]
These values can be verified by a computation in CoCoA \cite{Cocoa}.

\item We provide an ideal such that $\mu(I^k)$ has two local maxima and three local minima. With the notation of \cref{prop:linear-solution,thm:ABK-construction}, let $m=5$ and $\lambda=5$. Then
\[
a_2=9,\qquad a_3=2,\qquad a_4=15,\qquad a_5=19.
\]
Hence,
\[
p_1=9600,\qquad p_2=960,\qquad p_3=3200,\qquad p_4=600,\qquad p_5=480.
\]
Thus,
\[
\begin{aligned}
I_1 &= (x^{9600},y^{9600})(x^{57600},y^{57600}),\\
I_2 &= x^{20160}y^{48960}(x^{960},y^{960})^8,\\
I_3 &= x^{32000}y^{42560}(x^{3200},y^{3200}),\\
I_4 &= x^{39000}y^{33560}(x^{600},y^{600})^{14},\\
I_5 &= x^{48480}y^{24440}(x^{480},y^{480})^{18}.
\end{aligned}
\]
Let $I=I_1+I_2+I_3+I_4+I_5$. Then, by \eqref{number}, we obtain
\[
\mu(I)=49,\quad
\mu(I^2)=61,\quad
\mu(I^3)=49,\quad
\mu(I^4)=61,\quad
\mu(I^5)=36,\quad
\mu(I^6)=43.
\]
These values can be verified by a computation in CoCoA \cite{Cocoa}. Since
\[
49<61>49<61>36<43,
\]
the sequence $\mu(I^k)$ has two local maxima and three local minima.

\item Let $m=7$ and $\lambda=60$. Then
\[
a_2=72,\qquad a_3=5,\qquad a_4=37,\qquad a_5=41,\qquad a_6=51, \qquad a_7=271.
\]
Then, by \eqref{number}, we obtain
\[
\mu(I)=481,\quad
\mu(I^2)=421,\quad
\mu(I^3)=481,\quad
\mu(I^4)=481,\quad
\mu(I^5)=421,\quad
\mu(I^6)=481,\quad
\mu(I^7)=64.
\]
We see that the sign pattern of $\mu(I^{k+1})-\mu(I^{k})$ is $\{-,+, 0, -, +,-\}$. 
\end{enumerate}
\end{Example}

Recall that for an $\mathfrak m$-primary ideal $J$ in a Noetherian local ring $(R,\mathfrak m)$, the \emph{index of reducibility} of $J$ is
\[
r(R/J)=\ell_R\big((J:_R\mathfrak m)/J\big),
\]
which coincides with the Cohen--Macaulay type of $R/J$. By a theorem of Cuong, Quy, and Truong~\cite[Theorem~4.1]{CQT}, the function
\[
n\longmapsto r(R/J^n)
\]
is eventually a polynomial of degree $\dim R-1$. In the case of a two-dimensional regular local ring, \cite[Proposition~2.9]{AK} shows that
\[
\mu(J^n)=r(R/J^n)+1
\]
for all $n\ge 1$. Therefore, \cref{cor:terminal-negative} immediately yields the following consequence.

\begin{Corollary}\label{cor:Index-red}
Let $R=K[[x,y]]$ be the formal power series ring over a field $K$ with the unique maximal ideal $\fkm=(x,y)$. Let $m\ge 3$, and let
\[
\varepsilon_1,\ldots,\varepsilon_{m-1}\in\{+,-,0\}
\]
with $\varepsilon_{m-1}=-$. Then there exists an $\mathfrak m$-primary ideal $J\subset R$ such that
\[
\operatorname{sign}\bigl(r(R/J^{k+1})-r(R/J^k)\bigr)=\varepsilon_k
\qquad (1\le k\le m-1),
\]
and
\[
r(R/J^k)=(m+2)k
\qquad (k\ge m).
\]
\end{Corollary}


\begin{Acknowledgments}
The first author acknowledges support from LUMS under grant number STG-0240.
The second author was supported by JSPS KAKENHI Grant Number JP24K16909. 
\end{Acknowledgments}


\end{document}